\documentclass[a4paper]{article}
\usepackage{fullpage}
\usepackage{amsmath}
\usepackage{amssymb}
\usepackage{amsthm}
\usepackage{graphicx}
\usepackage{bbm}
\usepackage{enumerate}
\usepackage{microtype}
\usepackage{tocbibind}
\usepackage{xcolor}
\usepackage[none]{hyphenat}
\usepackage[colorlinks=true,urlcolor=blue,linkcolor=blue,plainpages=false,pdfpagelabels]{hyperref}
\allowdisplaybreaks

\newtheorem{theorem}{Theorem}[section]
\newtheorem{proposition}[theorem]{Proposition}

\newtheorem{corollary}[theorem]{Corollary}
\newtheorem{remark}[theorem]{Remark}

\newtheorem{definition}[theorem]{Definition}

\newcommand{\N}{\mathbb{N}}

\newcommand{\eps}{\varepsilon}

\title{On quantitative metastability for accretive operators}
\author{Andrei Sipo\c s${}^{a,b,c}$\\[2mm]
\footnotesize ${}^a$Research Center for Logic, Optimization and Security (LOS), Department of Computer Science,\\
\footnotesize Faculty of Mathematics and Computer Science, University of Bucharest,\\
\footnotesize Academiei 14, 010014 Bucharest, Romania\\[1mm]
\footnotesize ${}^b$Simion Stoilow Institute of Mathematics of the Romanian Academy,\\
\footnotesize Calea Grivi\c tei 21, 010702 Bucharest, Romania\\[1mm]
\footnotesize ${}^c$Institute for Logic and Data Science,\\
\footnotesize Popa Tatu 18, 010805 Bucharest, Romania\\[2mm]
\footnotesize Email: andrei.sipos@fmi.unibuc.ro\\
}
\date{}

\begin{document}

\maketitle

\begin{abstract}
Kohlenbach and the author have extracted a rate of metastability for approximating curves associated to continuous pseudocontractive self-mappings in Banach spaces which are uniformly convex and uniformly smooth, whose convergence is due to Reich. In this note, we show that this result may be extended to Reich's original convergence statement involving resolvents of accretive operators.


\end{abstract}

\section{Introduction}

If $T$ is a nonexpansive -- or, more generally, continuously pseudocontractive -- self-mapping of a convex, closed, bounded, nonempty subset of a uniformly smooth Banach space, and one denotes, for any $x$ in that set and any $t \in (0,1)$ by $x_t$ the unique point such that
$$x_t=tTx_t+(1-t)x,$$
then one has that the `approximating curve' $(x_t)$ converges to a fixed point of $T$ as $t \to 1$. This was first shown by Reich in 1980 \cite{Rei80}, thirteen years after this result had been proven for Hilbert spaces independently by Browder \cite{Bro67A} and Halpern \cite{Hal67}, thirteen years in which no such result was known to be true in any $L^p$ space other than the $L^2$ ones.

In a recent paper \cite{KohSip21}, Kohlenbach and the author have extracted a `rate of metastability' for the above result, and we shall now detail what we mean by that. That paper falls into the research program of `proof mining', which aims to analyze proofs in mainstream mathematics using tools from mathematical logic (interpretative proof theory) in order to extract additional (usually computational) content which may not be immediately apparent (for more details, see the book \cite{Koh08} and the recent survey \cite{Koh18}). Such additional content for a convergence theorem like the one above would naturally be a rate of convergence, but, in this case, counterexamples (due to e.g. Neumann \cite{Neu15}) show that a computable rate cannot exist even in Euclidean spaces. The next best thing, that in most cases theoretical results -- `metatheorems' -- of proof mining guarantee to be extractable is the above-mentioned rate of metastability  -- in the sense of Terence Tao 
\cite{Tao08,Tao08A}, the name having been suggested to him by Jennifer Chayes -- which is an upper bound on the $N$ in terms of the $\eps$ and of the $g$ (and possibily of more parameters specific to the problem at hand) in the following non-constructively equivalent formulation of the Cauchy property of a sequence:
$$\forall \eps > 0 \,\forall g\in\N^{\N} \,\exists N\in \N \,
\forall m,\ n\in [N,N+g(N)] \ \left(\| x_n-x_m\| \leq\eps\right).$$
The work of Kohlenbach and the author has shown that, by adding the additional hypothesis that the space is {\bf uniformly convex} (and thus still covering the case of $L^p$ spaces), the proof may be simplified to a sufficiently tame form that it falls under the guarantees of the metatheorems (see \cite{KohSip21} for more details), and thus that a rate of metastability may be obtained from the modified proof, which the paper's authors proceed to do (the question of finding such a rate had stood as an open problem in proof mining for ten years after similar rates had been found at the level of Hilbert spaces in \cite{Koh11, KohLeu12A}). The extracted bound is at once complicated (reflecting the manifold ways the various facets of the proof are put together) and complex (featuring the use of functional recursion going beyond primitive recursion), representing one of the most intricate rates of metastability ever produced by the proof mining program (and also a highly intricate formula by the standards of a general mathematical paper). As the authors of that paper point out in its introduction, `the enormous complexity of the final bound reflects the profound  combinatorial and computational content of Reich's deep theorem' \cite[p. 5]{KohSip21} (also, see that paper's introduction for more information on the historical significance of Reich's result and of the characterization of the limit point as the sunny nonexpansive retraction onto the fixed point set of $T$).

However, even disregarding the restriction to uniformly convex spaces, Reich's original result in \cite{Rei80} is more general than the one which has been analyzed above. Specifically, he states that, in the same background, if $A$ is an accretive operator satisfying some range condition (recovering the particular case above by taking $A:=Id-T$), then for any appropriate point $x$, the `approximating curve' $(J_{\lambda A}x)$ converges to a fixed point of $T$ as $\lambda \to \infty$. Accretive operators were introduced independently in 1967 by Browder \cite{Bro67} and Kato \cite{Kato}, as one of the natural generalizations of Hilbert spaces' monotone operators, and are motivated by the fact that many physical phenomena may be modelled by evolution equations of the form

$$\frac{\mathrm{d}x}{\mathrm{d}t} + Ax= 0.$$

{\it What we do in this paper is to show that the arguments in the paper \cite{KohSip21} may be generalized to the convergence of resolvents of accretive operators.}

The modifications that need to be made are non-trivial, but concentrated into some discrete regions of the proof. For that reason, and also due to the excessive length of the proof, we have chosen to focus our paper on the modifications themselves, referring frequently to the arguments and presentation of \cite{KohSip21} and only detailing the portions which are, relatively speaking, significantly changed.

Our result even improves slightly on the original one in the sense that a modulus of uniform continuity for $T$ is no longer needed. This is due to the fact that the resolvent of $A$ not only replaces the use of the operator $h_T$ (which was already, as it will be seen, a particular case of the resolvent), but also some (in hindsight extraneous) uses of $T$ itself. This is largely in line with the metatheorems for accretive operators recently obtained by Pischke \cite{Pis22}. Another small improvement is due to the fact that a careful examination of the proof shows that a bound on the point $x$ is no longer needed, and instead one just needs the bound on the diameter of the set $C$ (or on the distance from $x$ to some zero of $A$ in the `unbounded' versions of the theorem).

Section~\ref{prelim} reviews the necessary facts from the preliminaries section in \cite{KohSip21}, adding a new sub-section on accretive operators. Section~\ref{modifs} details the new versions of the qualitative convergence result, and how the increasingly quantitative proofs of it in \cite{KohSip21} need to be modified to accommodate the more general case of accretive operators. Section~\ref{final} presents and discusses our final, quantitative theorems.

We shall use the notation $\N^*:=\N \setminus \{0\}$.

\section{Preliminaries}\label{prelim}

\subsection{Banach spaces}

As in \cite{KohSip21}, we shall express and use uniform convexity for a Banach space $X$ in terms of the existence of a modulus $\eta : (0,2] \to (0,1]$, having the property that for all $\eps \in (0,2]$ and all $x$, $y \in X$ with $\|x\|\leq 1$, $\|y\| \leq 1$ and $\|x-y\| \geq \eps$ one has that
$$\left\|\frac{x+y}2\right\| \leq 1-\eta(\eps).$$

The following proposition has been obtained in \cite{KohSip21} as an application of \cite[Proposition~3.2]{BacKoh18}, which is a quantitative version of a theorem of Z\u alinescu \cite[Theorem~4.1]{Zal83}.

\begin{proposition}[{\cite[Proposition 2.4]{KohSip21}}]\label{l3-a}
Let $X$ be a uniformly convex Banach space having $\eta$ as a modulus and let $b \geq \frac12$. Put, for all $\eps \in (0,2]$,
$$\psi_{b,\eta}(\eps):= \min \left( \frac{\left(\min\left(\frac\eps2, \frac{\eps^2}{72b}\eta^2\left(\frac\eps{2b}\right)\right)\right)^2}4, \frac{\eps^2}{48}\eta^2\left(\frac{\eps}{2b}\right) \right).$$
Then, for all $\eps \in (0,2]$:
\begin{enumerate}[(a)]
\item $\psi_{b,\eta}(\eps) > 0$.
\item for all $x$, $y \in X$ with $\|x\|\leq b$, $\|y\| \leq b$, $\|x-y\| \geq \eps$, we have that
$$\left\| \frac{x+y}2 \right\|^2 + \psi_{b,\eta}(\eps) \leq \frac12\|x\|^2 + \frac12\|y\|^2.$$
\end{enumerate}
\end{proposition}

We shall use the single-valued normalized duality mapping $j:X \to X^*$ which exists in all smooth spaces, with its usual properties, and we shall denote, for all spaces $X$, all $x^* \in X^*$ and $y \in X$, $x^*(y)$ by $\langle y,x^* \rangle$. More information about duality mappings may be found in the book \cite{Cio90}.

Analogously to uniform convexity, and again as in \cite{KohSip21}, we shall express and use uniform smoothness for a Banach space $X$ in terms of the existence of a modulus $\tau : (0, \infty) \to (0, \infty)$ such that for all $\eps > 0$ and all $x$, $y \in X$ with $\|x\|= 1$, $\|y\| \leq \tau(\eps)$ one has that
$$\|x+y\| + \|x-y\| \leq 2 + \eps\|y\|.$$

It is known that uniform smoothness implies the the norm-to-norm uniform continuity on bounded subsets of the normalized duality mapping (in addition, as pointed out in \cite{KohSip21}, it has been proven independently by B\'enilan \cite[p. 0.5, Proposition 0.3]{Ben72} and K\"ornlein \cite[Appendix A]{Kor15} that the norm-to-norm uniform continuity on bounded subsets of an arbitrary duality selection mapping implies back uniform smoothness; see also the related characterization in \cite[Lemma 2.2 and the Remark on page 116]{Rei81}); the following proposition, first obtained in \cite{KohLeu12}, expresses this fact quantitatively.

\begin{proposition}[{\cite[Proposition 2.12]{KohSip21}}]
Let $X$ be a uniformly smooth Banach space with modulus $\tau$. Put, for all $b > 0$ and $\eps > 0$,
$$r_1(\eps):=\min(\eps,2), \qquad r_2(b):= \max(b,1), \qquad \omega_\tau(b,\eps) := \frac{r_1(\eps)^2}{12r_2(b)} \cdot \tau\left(\frac{r_1(\eps)}{2r_2(b)}\right).$$
Then for all $b >0$, $\eps >0$ and all $x$, $y \in X$ with $\|x\| \leq b$ and $\|y \| \leq b$, if $\|x-y\| \leq \omega_\tau(b,\eps)$ then $\|j(x)-j(y)\| \leq \eps$.
\end{proposition}

For the remainder of this section, we fix a smooth Banach space $X$ and $C \subseteq X$ a closed, convex, nonempty subset.

\subsection{Nonexpansive mappings and sunny nonexpansive retractions}

\begin{definition}
A map $Q:C \to X$ is called {\bf nonexpansive} if for all $x$, $y \in C$, $\|Qx-Qy\|\leq\|x-y\|$.
\end{definition}

Let $E \subseteq C$ be nonempty.

\begin{definition}
A map $Q:C\to E$ is called a {\bf retraction} if for all $x \in E$, $Qx=x$.
\end{definition}

\begin{definition}
A retraction $Q:C \to E$ is called {\bf sunny}\footnote{The term was first introduced in the paper \cite{Rei73}.} if for all $x \in C$ and $t \geq 0$ such that $Qx+t(x-Qx)\in C$,
$$Q(Qx+t(x-Qx))=Qx.$$
\end{definition}

\begin{proposition}[{\cite[Lemma 1.13.1]{GoeRei84}}]\label{char-sunny}
Let $Q: C\to E$ be a retraction. Then $Q$ is sunny and nonexpansive if and only if for all $x\in C$ and $y \in E$,
$$\langle x-Qx,j(y-Qx)\rangle \leq 0.$$
\end{proposition}

\begin{proposition}
There is at most one sunny nonexpansive retraction from $C$ to $E$.
\end{proposition}

\begin{proof}
See \cite[Proposition 2.17]{KohSip21}.
\end{proof}

More information about sunny nonexpansive retractions may be found in the paper \cite{KopRei}.

\subsection{Pseudocontractions}

\begin{definition}[{\cite[Definition 1]{Bro67}}]
A map $T:C \to C$ is called a {\bf pseudocontraction} if for all $x$, $y \in C$ and $t >0$, we have that
\begin{equation}
t\|x-y\| \leq \|(t+1)(x-y) - (Tx-Ty)\|.\label{def-psc}
\end{equation}
\end{definition}

\begin{proposition}
Any nonexpansive map is a pseudocontraction.
\end{proposition}

\begin{proof}
See \cite[Proposition 2.21]{KohSip21}.
\end{proof}

We have the following equivalence.

\begin{proposition}[{\cite[Proposition 1]{Bro67}}]
Let $T: C\to C$. Then $T$ is a pseudocontraction if and only if for all $x$, $y \in C$,
$$\langle Tx-Ty,j(x-y) \rangle \leq \|x-y\|^2.$$
\end{proposition}

\begin{definition}[{cf. \cite[(2.9)]{Gus67}}]
Let $k \in (0,1)$. We say that a map $T: C \to C$ is a {\bf $k$-strong pseudocontraction} if for all $x$, $y \in C$,
$$\langle Tx-Ty,j(x-y) \rangle \leq k\|x-y\|^2.$$
\end{definition}

\begin{proposition}\label{sps1}
Let $T: C \to C$ be a continuous pseudocontraction, $k \in (0,1)$ and $u \in C$. Define the map $U: C \to C$, by putting, for all $x \in C$, $Ux:=kTx+(1-k)u$. Then $U$ is a continuous $k$-strong pseudocontraction.
\end{proposition}

\begin{proof}
See \cite[Proposition 2.24]{KohSip21}.
\end{proof}

\begin{proposition}\label{sps2}
Let $k \in (0,1)$ and $T: C \to C$ be a continuous $k$-strong pseudocontraction. Then $T$ has a unique fixed point.
\end{proposition}

\begin{proof}
If $x$ and $y$ are fixed points of $T$, $\|x-y\|^2\leq k\|x-y\|^2$, so $x=y$. The existence of a fixed point follows from \cite[Proposition 3]{Mar73} (or \cite[Corollary 1]{Dei}) and the convexity of $C$.
\end{proof}

\subsection{Accretive operators}

\begin{definition}
An operator $A \subseteq X \times X$ is called {\bf accretive} if one of the following equivalent conditions holds (the equivalence follows from the classical lemma of Kato \cite[Lemma 1.1]{Kato}):
\begin{itemize}
\item for all $\lambda>0$ and all $(x_1,y_1)$, $(x_2,y_2)\in A$,
$$\|x_1 -x_2 \| \leq \|x_1 - x_2 + \lambda (y_1-y_2)\|;$$
\item for all $(x_1,y_1)$, $(x_2,y_2)\in A$,
$$\langle y_1-y_2, j(x_1-x_2) \rangle \geq 0.$$
\end{itemize}
\end{definition}

\begin{remark}
If $A \subseteq X \times X$ is accretive and $\lambda >0$, then $\lambda A$ is also accretive.
\end{remark}

\begin{definition}
An operator $A \subseteq X \times X$ is called {\bf $m$-accretive} if it is accretive and for all $\lambda >0$, $ran(Id+\lambda A)=X$.
\end{definition}

\begin{remark}
If $A \subseteq X \times X$ is $m$-accretive and $\lambda >0$, then $\lambda A$ is also $m$-accretive.
\end{remark}

\begin{definition}
If $A \subseteq X \times X$ is accretive, one defines the {\bf resolvent} of $A$, $J_A : ran(Id+A) \to X$, for every $x \in ran(Id+A)$, by setting $J_Ax$ to be the unique $y \in X$ such that there is a $z \in X$ with $(y,z) \in A$ and $y+z=x$.
\end{definition}

\begin{remark}
If $A \subseteq X \times X$ is accretive, then $J_A$ is nonexpansive and its fixed points coincide with the zeros of $A$.
\end{remark}

From now on, if we have a mapping $T:C \to C$, by $A:=Id-T$ we shall mean 
$$A:=\{(x,x-Tx) \mid x \in C\}.$$

\begin{remark}
Let $T: C \to C$ be a pseudocontraction and $A:=Id-T$. Then $A$ is accretive and the zeros of $A$ coincide with the fixed points of $T$.
\end{remark}

\begin{proposition}\label{psc-acc}
Let $T: C \to C$ be a pseudocontraction and $A:=Id-T$. Let $x \in C$ and $\lambda >0$. Set $t:=\frac{\lambda}{1+\lambda} \in (0,1)$ and let $x_t \in C$ be such that $x_t = tTx_t + (1-t)x$ (if $T$ is continuous, a unique such point exists by Propositions~\ref{sps1} and \ref{sps2}). Then $x \in ran(Id+\lambda A)$ and $J_{\lambda A}x =x_t$.
\end{proposition}

\begin{proof}
Since
$$x_t=\frac{\lambda}{1+\lambda}Tx_t + \frac{1}{1+\lambda}x ,$$
we have that
$$(1+\lambda)x_t=\lambda Tx_t + x ,$$
so
$$x_t+\lambda(x_t-Tx_t) = x.$$
Since $(x_t, x_t-Tx_t) \in A$, the conclusion follows.
\end{proof}

\begin{corollary}\label{psc-acc-cor}
Let $T: C \to C$ be a continuous pseudocontraction and $A:=Id-T$. Let $\lambda >0$. Then $C\subseteq ran(Id+\lambda A)$ and, for all $y \in C$, $J_{\lambda A}y \in C$.
\end{corollary}

More information on accretive operators and their resolvents can be found in the paper \cite{Rei80b}.

\section{Modifications to the proofs}\label{modifs}

The main qualitative theorem that we focus on is the following (note that, compared to \cite{KohSip21}, we have removed the bounding condition on the elements of $C$, leaving just the bounding of the diameter, as said condition does not feature at all in the proof).

\begin{theorem}[{cf. \cite{Rei80}; see also \cite{BruRei}}]\label{main-thm}
Let $X$ be a Banach space which is uniformly convex and uniformly smooth. Let $C \subseteq X$ be a closed, convex, bounded, nonempty subset. Let $A \subseteq X \times X$ be an accretive operator such that, for all $\lambda > 0$, we have that $C\subseteq ran(Id+\lambda A)$ and that, for all $y \in C$, $J_{\lambda A}y \in C$. Let $x \in C$.

Then, for all $(\lambda_n) \subseteq (0,\infty)$ such that $\lim\limits_{n \to \infty} \lambda_n = \infty$, we have that $(J_{\lambda_n A} x)$ is Cauchy.
\end{theorem}

We first list some corollaries, which show how one can weaken the boundedness constraint and how one can recover the result for pseudocontractions as stated in \cite{KohSip21} (where, in particular, we shall no longer need uniform continuity, either qualitatively or quantitatively, but just the plain assumption of continuity).

\begin{corollary}\label{cor-ubou}
Let $X$ be a Banach space which is uniformly convex and uniformly smooth. Let $C \subseteq X$ be a closed, convex, nonempty subset. Let $A \subseteq X \times X$ be an accretive operator such that, for all $\lambda > 0$, we have that $C\subseteq ran(Id+\lambda A)$ and that, for all $y \in C$, $J_{\lambda A}y \in C$ (in particular, if $C=X$, then $A$ is $m$-accretive). Let $x \in C$. Let $p \in C$ be a zero of $A$.

Then, for all $(\lambda_n) \subseteq (0,\infty)$ such that $\lim\limits_{n \to \infty} \lambda_n = \infty$, we have that $(J_{\lambda_n A} x)$ is Cauchy.
\end{corollary}

\begin{proof}
Let $b \in \N^*$ be such that $\|x-p\| \leq b$. Put $D:= C \cap \overline{B}(p,b)$, so $x \in D$. We shall apply Theorem~\ref{main-thm} with $C:=D$, from which we will get our conclusion. The only non-trivial condition to check is that, for all $\lambda >0$ and $y \in D$, $J_{\lambda A}y \in D$. Let $\lambda>0$ and $y \in D$. Then $y \in C$, so $J_{\lambda A}y \in C$. Since $\|J_{\lambda A}y - p \|\leq \|y-p\| \leq b$, we have that $J_{\lambda A}y\in \overline{B}(p,b)$, so $J_{\lambda A}y \in D$.
\end{proof}

\begin{corollary}
Let $X$ be a Banach space which is uniformly convex and uniformly smooth. Let $C \subseteq X$ be a closed, convex, bounded, nonempty subset. Let $T: C \to C$ be a pseudocontraction and $x \in C$. For all $t \in (0,1)$, let $x_t \in C$ be such that $x_t = tTx_t + (1-t)x$ (if $T$ is continuous, a unique such point exists by Propositions~\ref{sps1} and \ref{sps2}). Then for all $(t_n) \subseteq (0,1)$ such that $\lim\limits_{n \to \infty} t_n = 1$ we have that $(x_{t_n})$ is Cauchy.
\end{corollary}

\begin{proof}
Set $A:=Id-T$. Set, for all $n$, $\lambda_n:= \frac{t_n}{1-t_n} \in (0,\infty)$, so $t_n = \frac{\lambda_n}{1+\lambda_n}$. Clearly, $\lim\limits_{n \to \infty} \lambda_n = \infty$. Using Proposition~\ref{psc-acc} and Corollary~\ref{psc-acc-cor}, we get that for all $n$, $x_{t_n}=J_{\lambda_n A} x$ and that all the other conditions in Theorem~\ref{main-thm} are satisfied.
\end{proof}

\begin{corollary}
Let $X$ be a Banach space which is uniformly convex and uniformly smooth. Let $C \subseteq X$ be a closed, convex, nonempty subset.  Let $T: C \to C$ be a pseudocontraction and $x \in C$. For all $t \in (0,1)$, let $x_t \in C$ be such that $x_t = tTx_t + (1-t)x$ (if $T$ is continuous, a unique such point exists by Propositions~\ref{sps1} and \ref{sps2}). Let $p \in C$ be a fixed point of $T$. Then for all $(t_n) \subseteq (0,1)$ such that $\lim\limits_{n \to \infty} t_n = 1$ we have that $(x_{t_n})$ is Cauchy.  
\end{corollary}

\begin{proof}
Set $A:=Id-T$. Set, for all $n$, $\lambda_n:= \frac{t_n}{1-t_n} \in (0,\infty)$, so $t_n = \frac{\lambda_n}{1+\lambda_n}$. Clearly, $\lim\limits_{n \to \infty} \lambda_n = \infty$. Using Proposition~\ref{psc-acc} and Corollary~\ref{psc-acc-cor}, we get that for all $n$, $x_{t_n}=J_{\lambda_n A} x$ and that all the other conditions in Corollary~\ref{cor-ubou} are satisfied.
\end{proof}

In the paper \cite{KohSip21}, multiple proofs (which are progressively made more quantitative) are presented of the main result, and in the following subsections, we sketch the modifications that have to be made to said proofs in order to adapt them to proving our main result, Theorem~\ref{main-thm}. Throughout, we shall take $b \in \N^*$ such that the diameter of $C$ is bounded by $b$, and $\eta$ and $\tau$ be moduli of convexity and smoothness (respectively) for $X$.

\subsection{Modifications to the original proof of Morales}\label{morales}

The first proof presented in \cite{KohSip21} is actually the proof which started the quantitative investigations in that paper, and is due to Morales \cite{Mor90} (in fact, the original argument of Reich \cite{Rei80} was not too different, but since he used the actual limit and not the limsup like here or the Banach limit like in some other authors' work, he had to first restrict himself to some separable subspace in order to make said limit exist).

Throughout, one has to replace $(t_n)$ by $(\lambda_n)$ -- and denote, where appropriate, for all $n$, $x_n:=J_{\lambda_n A} x$ -- and also replace fixed points of $T$ by zeros of $A$ (or equivalently fixed points of $J_A$). Then, instead of proving the `asymptotic regularity' or `approximate fixed point' statement of $\lim_{n \to \infty} \|x_n -Tx_n \| = 0$, or equivalently $\lim_{n \to \infty} \|x_n -h_Tx_n \| = 0$, one shows that $\lim_{n \to \infty} \|x_n -J_Ax_n \| = 0$. The argument goes as follows: for all $n$, since $\left(x_n, \frac{x-x_n}{\lambda_n}\right) \in A$ and $(J_A x_n, x_n - J_Ax_n) \in A$, and $A$ is accretive, one has that
$$\|x_n - J_Ax_n \| \leq \left\|x_n - J_Ax_n + \frac{x-x_n}{\lambda_n} - (x_n - J_Ax_n) \right\| = \left\| \frac{x-x_n}{\lambda_n}  \right\| \leq \frac{b}{\lambda_n},$$
so $\lim_{n \to \infty} \|x_n -J_Ax_n \| = 0$. 

One then more generally replaces $h_T$ by $J_A$ in the course of the proof (in fact, in the case that $A$ is of the form $Id-T$, one can check that the mapping $h_T$ used in \cite{KohSip21} is actually the same as $J_A$, so the use of $J_A$ is not that surprising).

Then, in order to obtain the first convergence statement in the proof, instead of using the fact that $T$ is a pseudocontraction, one uses the accretivity of $A$, as follows. Since $\left(x_n, \frac{x-x_n}{\lambda_n}\right) \in A$ and $(p,0) \in A$, and $A$ is accretive, one has that
$$\left\langle \frac{x-x_n}{\lambda_n} - 0, j(x_n-p) \right\rangle \geq 0,$$
i.e. that $\langle x_n - x, j(x_n -p)\rangle \leq 0$, which one then sums up with the previously obtained inequality $\limsup_{k \to \infty} \langle x - p, j(x_{n_k} - p) \rangle \leq 0$ to get that $(x_{n_k}) \to p$.

The proof then continues in the same way, with the appropriate modifications.

\subsection{\texorpdfstring{Modifications to the proof using limsup's but only $\varepsilon$-infima}{Modifications to the proof using limsup's but only epsilon-infima}}\label{llimsup}

We consider this proof to start more or less with the `Second Proof of the Claim' in \cite{KohSip21}. Claims 1 and 2 (together with the second `wrap-up' of the proof of the theorem) do not need many changes (but those changes are consequential, as we shall see momentarily). Again, one uses $J_A$ instead of $h_T$, but also instead of $T$ for `approximate fixed point' statements, e.g. $\|u-Tu\| \leq \eps$ becomes $\|u-J_Au\| \leq \eps$ and $\|u_m-Tu_m\| \leq 1/m$ becomes $\|u_m-J_Au_m\| \leq 1/m$. Since $J_A$ does the job for both, one does not need to `oscillate' between $T$ and $h_T$, and thus does not need to use the modulus of uniform continuity $\theta$. Therefore, one can take
$$\eta_1:= \min\left(\eps, \frac12 \psi_{b,\eta}(\eps) \right) > 0.$$

The most elaborate changes are in Claim 3, and thus we shall express the `new' portion in detail. To orient ourselves, we give now the new statement of this claim: for all $(\lambda_n) \subseteq (0,\infty)$ such that $\lim\limits_{n \to \infty} \lambda_n = \infty$ and for all $\eps > 0$ there is a $v \in C$ such that:
\begin{itemize}
\item for all $z \in C$, $\limsup\limits_{n \to \infty} \|J_{\lambda_n A} x-v\|^2 \leq \limsup\limits_{n \to \infty} \|J_{\lambda_n A} x-z\|^2 + \eps$;
\item for all $\lambda\in(0,\infty)$, $\langle J_{\lambda A} x-x, j(J_{\lambda A} x-v) \rangle \leq \eps$.
\end{itemize}

To prove this claim, we presuppose the truth of the new version of Claim 2, i.e. that for all $(\lambda_n) \subseteq (0,\infty)$ such that $\lim\limits_{n \to \infty} \lambda_n = \infty$ and for all $\eps > 0$ there is a $u \in C$ such that:
\begin{itemize}
\item for all $z \in C$, $\limsup\limits_{n \to \infty} \|J_{\lambda_n A} x-u\|^2 \leq \limsup\limits_{n \to \infty} \|J_{\lambda_n A} x-z\|^2 + \eps$;
\item $\|u-J_Au\| \leq \eps$.
\end{itemize}

Take $\eta_2:=\min\left(\eps, \frac12\psi_{b,\eta}\left(\omega_\tau\left(b,\frac{\eps}{2b}\right)\right)\right).$
Apply (our version of) Claim 2 for $(\lambda_n)$ and $\eta_2$ and put $v$ to be the resulting $u$.
We have to show that for all $\lambda\in(0,\infty)$, $\langle J_{\lambda A} x-x, j(J_{\lambda A} x-v) \rangle \leq \eps$. Let $\lambda \in (0,\infty)$ and put $\delta:=\min\left(\eta_2,\frac{\eps}{4b\lambda},\omega_\tau(b,\frac{\eps}{4b})\right)$. Apply (our version of) Claim 2 for $(\lambda_n)$ and $\delta$ and put $v'$ to be the resulting $u$, so in particular $\|v'-J_Av'\|\leq\delta$. Since $\left(J_{\lambda A} x,\frac{x-J_{\lambda A} x}{\lambda}\right) \in A$ and $(J_Av', v'-J_Av') \in A$, we have that
$$\left\langle  \frac{x-J_{\lambda A} x}{\lambda} - ( v'-J_Av'), j(J_{\lambda A} x-J_Av') \right\rangle \geq 0,$$
so
$$\left\langle  \frac{x-J_{\lambda A} x}{\lambda}, j(J_{\lambda A} x-J_Av') \right\rangle \leq \langle   J_Av'- v', j(J_{\lambda A} x-J_Av') \rangle \leq \delta b,$$
from which we get that
\begin{equation}\label{e1j}
\langle J_{\lambda A} x-x, j(J_{\lambda A} x-J_Av') \rangle \leq \delta b\lambda \leq \frac{\eps}4.
\end{equation}
On the other hand, we know that $\|v'-J_Av'\| \leq \delta \leq \omega_\tau(b,\frac{\eps}{4b})$, i.e. that $\|(J_{\lambda A} x-v') - (J_{\lambda A} x-J_Av')\| \leq \omega_\tau(b,\frac{\eps}{4b})$, so
$$\|j(J_{\lambda A} x-v') - j(J_{\lambda A} x-J_Av')\| \leq \frac\eps{4b},$$
from which we get that
\begin{equation}\label{e2j}
\langle J_{\lambda A} x-x, j(J_{\lambda A} x-v')-j(J_{\lambda A} x-J_Av')\rangle \leq \frac\eps4.
\end{equation}
From \eqref{e1j} and \eqref{e2j} we get that
$$\langle J_{\lambda A} x-x, j(J_{\lambda A} x-v') \rangle \leq  \frac{\eps}2.$$

The proof then continues in the same way until the end of Claim 3, and all the way to the end of Claim 7, with the usual modifications. 
 
\subsection{Modifications to the proof using approximate limsup's}

We first remark that all the routine modifications that we have applied to the previous proof are still in effect, e.g. the use of $h_T$ is replaced by that of $J_A$.

Since we have also replaced the sequence $(t_n)$ by a sequence $(\lambda_n)$, we have to introduce new moduli to govern that sequence's behaviour. We shall thus consider now moduli $\alpha : \N \to \N$ and $\gamma : \N \to \N^*$ such that:
\begin{itemize}
\item for all $n$ and all $m \geq \alpha(n)$, $\lambda_m \geq n+1$;
\item for all $n$, $\lambda_n \leq \gamma(n)$.
\end{itemize}
In the case that, for all $n$, $\lambda_n = n+1$, we may take, for all $n$, $\alpha(n):=n$ and $\gamma(n):=n+1$.

The quantities $\beta$, $q$ and $\nu_1$ introduced afterwards are modified as follows:
\begin{align*}
  \beta(c,\eps) &:= \frac{p(\eps)}{4b\gamma(c)} \\
  q(p,c,d,\eps) &:= \min \left\{\beta(c,\eps), \beta(s_{p,g}(d),\eps), \omega_\tau\left(b,\frac{p(\eps)}{4b}\right)\right\}\\
  \nu_1(p,c,d,\eps) &:= \frac12\psi_{b,\eta}(q(p,c,d,\eps)).
\end{align*}

The proofs of Claims I and II go largely without change, and the only significant modifications to this proof arise in Claim III, namely in its first three sub-claims. We shall detail those now.

In the new version of Sub-claim 1, one has to prove that $\|x_{h}-v\|^2 - \left\|x_{h} -\frac{v+J_Av}2\right\|^2$, $\|x_{h}-J_Av\|^2 - \left\|x_{h} -\frac{v+J_Av}2\right\|^2$, $\|x_{h'}-v'\|^2 - \left\|x_{h'} -\frac{v'+J_Av'}2\right\|^2$, $\|x_{h'}-J_Av'\|^2 - \left\|x_{h'} -\frac{v'+J_Av'}2\right\|^2$ are all smaller or equal to $\nu_1(w,k,k',\eps)$.

To prove that, we first use the previously shown fact that
$$\|x_h - J_Ax_h \| \leq \frac{b}{\lambda_h},$$
so
$$\|x_h-J_Av\| \leq \|J_Av-J_Ax_h\| + \|x_h-J_Ax_h\| \leq \|x_h - v\| + \frac{b}{\lambda_h}.$$
Then we may write:
\begin{align*}
\|x_h - J_Av\|^2 &\leq \|x_h-v\| ^ 2 +\frac{b^2}{\lambda_h^2}+ 2b\cdot\frac{b}{\lambda_h}\\
&\leq \left\|x_h - \frac{v+J_Av}2 \right\|^2 + \frac{\nu_1(w,k,k',\eps)}2 + \frac{b^2}{\frac{4b^2}{\nu_1(w,k,k',\eps)}} + \frac{2b^2}{\frac{8b^2}{\nu_1(w,k,k',\eps)}} \\
&\leq \left\|x_h - \frac{v+J_Av}2 \right\|^2 + \frac{\nu_1(w,k,k',\eps)}2 + \frac{\nu_1(w,k,k',\eps)}4 + \frac{\nu_1(w,k,k',\eps)}4 \\
&=  \left\|x_h - \frac{v+J_Av}2 \right\|^2 + \nu_1(w,k,k',\eps).
\end{align*}
Similarly, we may show that $\|x_{h'}-J_Av'\|^2 - \left\|x_{h'} -\frac{v'+J_Av'}2\right\|^2  \leq \nu_1(w,k,k',\eps)$ and, thus, the proof of Sub-claim 1 is finished.

We then remark that the modulus of uniform continuity $\theta$ has disappeared from the new definition $\eta_1$, and this comes into play in the new version of Sub-claim 2, where one only shows that $\|J_Av-v\|,\ \|J_Av'-v'\| \leq q(w,k,k',\eps).$

In the new version of Sub-claim 3, one has to show that:
\begin{itemize}
\item $\langle x_k-x, j(x_k-w) \rangle$, $\langle x^w_{k'}-x,j(x^w_{k'}-w')\rangle \leq \frac{\nu_4(\eps)}3$;
\item $\langle x^w_{k'}-x,j(x^w_{k'}-w)\rangle$, $\langle x_k-x, j(x_k-w') \rangle \leq \frac{\eps^2}{96}$.
\end{itemize}

As with Claim 3 in the previous subsection, we only detail the beginning of the proof of this sub-claim.

We have that $\|J_Av-v\| \leq q(w,k,k',\eps)$. Since $\left(x_k,\frac{x-x_k}{\lambda_k}\right) \in A$ and $(J_Av, v-J_Av) \in A$, we have that
$$\left\langle  \frac{x-x_k}{\lambda_k} - ( v-J_Av), j(x_k-J_Av) \right\rangle \geq 0,$$
so
$$\left\langle  \frac{x-x_k}{\lambda_k}, j(x_k-J_Av) \right\rangle \leq \langle   J_Av- v, j(x_k-J_Av) \rangle \leq q(w,k,k',\eps) \cdot b,$$
from which we get that
\begin{equation}\label{e1j2}
\langle x_k-x, j(x_k-J_Av) \rangle \leq q(w,k,k',\eps) \cdot b\lambda_k\leq \beta(k,\eps) \cdot b\lambda_k = b\lambda_k \cdot \frac{p(\eps)}{4b\gamma(k)} \leq \frac{p(\eps)}4.
\end{equation}
On the other hand, we know that $\|v-J_Av\| \leq q(w,k,k',\eps) \leq \omega_\tau\left(b,\frac{p(\eps)}{4b}\right)$, i.e. that $\|(x_k-v) - (x_k-J_Av)\| \leq \omega_\tau\left(b,\frac{p(\eps)}{4b}\right)$, so
$$\|j(x_k-v) - j(x_k-J_Av)\| \leq \frac{p(\eps)}{4b},$$
from which we get that
\begin{equation}\label{e2j2}
\langle x_k-x, j(x_k-v)-j(x_k-J_Av)\rangle \leq \frac{p(\eps)}4.
\end{equation}
From \eqref{e1j2} and \eqref{e2j2} we get that
$$\langle x_k-x, j(x_k-v) \rangle \leq  \frac{p(\eps)}2.$$

The proof then continues in the same way until its end, and actually no more modifications need to be made to the original text in \cite{KohSip21}.

\subsection{Modifications to the majorization and the rate}

We now proceed to derive the actual new rate of metastability. We first remark that the `higher-order proof mining' performed in Section 5 of \cite{KohSip21} to extract the actual witness remains the same in its entirety. Only in Section 6, where all those quantities are `majorized', we need to operate some changes, and those only appear in the `majorant' of $\nu_1$, i.e. the function $\nu^*_1$, which becomes
$$\nu^*_1(m,n):=\frac12 \min_{c \leq \max\left(m,n+ g^M(n)\right)} \psi_{b,\eta}\left(\min \left\{\beta(c,\eps), \omega_\tau\left(b,\frac{p(\eps)}{4b}\right)\right\}\right),$$
noting that the function $\beta$ is also changed from the original, as per the previous sub-section.

Since, as remarked previously, the quantities no longer depend on the modulus of uniform continuity $\theta$, one may drop it from the indices, and thus denote the final extracted quantity by
$$\Theta_{b,\eta,\tau,\alpha,\gamma}(\eps,g).$$

\section{Final results and discussion}\label{final}

Summing up, we may express the main result of the present paper as follows.

\begin{theorem}\label{main-metastab}
Let $X$ be a Banach space which is uniformly convex with modulus $\eta$ and uniformly smooth with modulus $\tau$. Let $C \subseteq X$ be a closed, convex, nonempty subset. Let $b \in \N^*$ be such that  the diameter of $C$ is bounded by $b$. Let $A \subseteq X \times X$ be an accretive operator such that, for all $\lambda > 0$, we have that $C\subseteq ran(Id+\lambda A)$ and that, for all $y \in C$, $J_{\lambda A}y \in C$. Let $x \in C$.

Let $(\lambda_n) \subseteq (0,\infty)$, $\alpha : \N \to \N$ and $\gamma : \N \to \N^*$ be such that:
\begin{itemize}
\item for all $n$ and all $m \geq \alpha(n)$, $\lambda_m \geq n+1$;
\item for all $n$, $\lambda_n \leq \gamma(n)$.
\end{itemize}
Then, for all $\eps > 0$ and $g: \N \to \N$ there is an $N \leq \Theta_{b,\eta,\tau,\alpha,\gamma}(\eps,g)$ such that for all $m$, $n \in [N,N+g(N)]$,
$\|J_{\lambda_m A} x-J_{\lambda_n A} x\| \leq \eps.$
\end{theorem}

We shall now detail the corresponding quantitative corollaries.

\begin{corollary}\label{cor-metastab}
Let $X$ be a Banach space which is uniformly convex with modulus $\eta$ and uniformly smooth with modulus $\tau$. Let $C \subseteq X$ be a closed, convex, nonempty subset. Let $A \subseteq X \times X$ be an accretive operator such that, for all $\lambda > 0$, we have that $C\subseteq ran(Id+\lambda A)$ and that, for all $y \in C$, $J_{\lambda A}y \in C$ (in particular, if $C=X$, then $A$ is $m$-accretive). Let $x \in C$. Let $p \in C$ be a zero of $A$ and $b \in \N^*$ be such that $\|x-p\| \leq b/2$.

Let $(\lambda_n) \subseteq (0,\infty)$, $\alpha : \N \to \N$ and $\gamma : \N \to \N^*$ be such that:
\begin{itemize}
\item for all $n$ and all $m \geq \alpha(n)$, $\lambda_m \geq n+1$;
\item for all $n$, $\lambda_n \leq \gamma(n)$.
\end{itemize}
Then, for all $\eps > 0$ and $g: \N \to \N$ there is an $N \leq \Theta_{b,\eta,\tau,\alpha,\gamma}(\eps,g)$ such that for all $m$, $n \in [N,N+g(N)]$,
$\|J_{\lambda_m A} x-J_{\lambda_n A} x\| \leq \eps.$
\end{corollary}

\begin{proof}
Put $D:= C \cap \overline{B}(p,b/2)$, so $x \in D$. We shall apply Theorem~\ref{main-metastab} with $C:=D$, from which we will get our conclusion. The only non-trivial condition to check is that, for all $\lambda >0$ and $y \in D$, $J_{\lambda A}y \in D$. Let $\lambda>0$ and $y \in D$. Then $y \in C$, so $J_{\lambda A}y \in C$. Since $\|J_{\lambda A}y - p \|\leq \|y-p\| \leq b/2$, we have that $J_{\lambda A}y\in \overline{B}(p,b/2)$, so $J_{\lambda A}y \in D$.
\end{proof}

We now turn to instantiating the above results for continuous pseudocontractions.

In the following corollaries, we shall consider, for sequences $(t_n) \subseteq (0,1)$, moduli $\xi : \N \to \N$ and $\gamma : \N \to \N^*$ such that:
\begin{itemize}
\item for all $n$ and all $m \geq \xi(n)$, $t_m \geq 1 - \frac1{n+1}$;
\item for all $n$, $t_n \leq 1 - \frac1{\gamma(n)}$.
\end{itemize}
In the case that, for all $n$, $t_n = 1 - \frac1{n+2}$, we may take, for all $n$, $\xi(n):=n$ and $\gamma(n):=n+2$.

We also define, for any $\xi : N \to \N$, the function $\xi':\N \to \N$, where we set, for all $n$, $\xi'(n):=\xi(n+1)$. We also denote, for any choice of parameters,
$$\widetilde{\Theta}_{b,\eta,\tau,\xi,\gamma} := \Theta_{b,\eta,\tau,\xi',\gamma}.$$

\begin{corollary}
Let $X$ be a Banach space which is uniformly convex with modulus $\eta$ and uniformly smooth with modulus $\tau$. Let $C \subseteq X$ be a closed, convex, nonempty subset. Let $b \in \N^*$ be such that the diameter of $C$ is bounded by $b$. Let $T: C \to C$ be a pseudocontraction and $x \in C$. For all $t \in (0,1)$, let $x_t \in C$ be such that $x_t = tTx_t + (1-t)x$ (again, if $T$ is continuous, a unique such point exists by Propositions~\ref{sps1} and \ref{sps2}).

Let $(t_n) \subseteq (0,1)$, $\xi : \N \to \N$ and $\gamma : \N \to \N^*$ be such that:
\begin{itemize}
\item for all $n$ and all $m \geq \xi(n)$, $t_m \geq 1 - \frac1{n+1}$;
\item for all $n$, $t_n \leq 1 - \frac1{\gamma(n)}$.
\end{itemize}
Then, for all $\eps > 0$ and $g: \N \to \N$ there is an $N \leq \widetilde{\Theta}_{b,\eta,\tau,\xi,\gamma}(\eps,g)$ such that for all $m$, $n \in [N,N+g(N)]$,
$\|x_{t_m}-x_{t_n}\| \leq \eps.$
\end{corollary}

\begin{proof}
Set $A:=Id-T$. Set, for all $n$, $\lambda_n:= \frac{t_n}{1-t_n} \in (0,\infty)$, so $t_n = \frac{\lambda_n}{1+\lambda_n}$. Using Proposition~\ref{psc-acc} and Corollary~\ref{psc-acc-cor}, we get that for all $n$, $x_{t_n}=J_{\lambda_n A} x$, and that most of the conditions of Theorem~\ref{main-metastab}, which we seek to apply, are satisfied. We only need to check the moduli conditions for $\xi'$ and $\gamma$.

Firstly, let $n \in \N$ and $m \geq \xi'(n)=\xi(n+1)$. We want to show that $\lambda_m \geq n+1$. Since, by the condition on $\xi$, we have that $t_m \geq 1 - \frac1{n+2} = \frac{n+1}{n+2}$, and so that $1-t_m \leq \frac1{n+2}$, we get that $\lambda_m = \frac{t_m}{1-t_m} \geq \frac{n+1}{n+2} \cdot (n+2) = n+1$.

Let now $n \in \N$. We want to show that $\lambda_n \leq \gamma(n)$. Since, by the condition on $\gamma$, we have that $t_n \leq 1 - \frac1{\gamma(n)} = \frac{\gamma(n)-1}{\gamma(n)}$, and so that $1-t_n \geq \frac1{\gamma(n)}$, we get that $\lambda_n = \frac{t_n}{1-t_n} \leq \frac{\gamma(n)-1}{\gamma(n)} \cdot \gamma(n) = \gamma(n)-1\leq \gamma(n)$.

The conclusion then follows.
\end{proof}

\begin{corollary}
Let $X$ be a Banach space which is uniformly convex with modulus $\eta$ and uniformly smooth with modulus $\tau$. Let $C \subseteq X$ be a closed, convex, nonempty subset.  Let $T: C \to C$ be a pseudocontraction and $x \in C$. For all $t \in (0,1)$, let $x_t \in C$ be such that $x_t = tTx_t + (1-t)x$ (again, if $T$ is continuous, a unique such point exists by Propositions~\ref{sps1} and \ref{sps2}). Let $p \in C$ be a fixed point of $T$ and $b \in \N^*$ be such that $\|x-p\| \leq b/2$.

Let $(t_n) \subseteq (0,1)$, $\xi : \N \to \N$ and $\gamma : \N \to \N^*$ be such that:
\begin{itemize}
\item for all $n$ and all $m \geq \xi(n)$, $t_m \geq 1 - \frac1{n+1}$;
\item for all $n$, $t_n \leq 1 - \frac1{\gamma(n)}$.
\end{itemize}
Then, for all $\eps > 0$ and $g: \N \to \N$ there is an $N \leq \widetilde{\Theta}_{b,\eta,\tau,\xi,\gamma}(\eps,g)$ such that for all $m$, $n \in [N,N+g(N)]$,
$\|x_{t_m}-x_{t_n}\| \leq \eps.$
\end{corollary}

\begin{proof}
Set $A:=Id-T$. Set, for all $n$, $\lambda_n:= \frac{t_n}{1-t_n} \in (0,\infty)$, so $t_n = \frac{\lambda_n}{1+\lambda_n}$. Using Proposition~\ref{psc-acc} and Corollary~\ref{psc-acc-cor}, we get that for all $n$, $x_{t_n}=J_{\lambda_n A} x$, and that most of the conditions of Corollary~\ref{cor-metastab}, which we seek to apply, are satisfied. The moduli conditions are checked as in the previous proof.
\end{proof}

We now argue, as in \cite{KohSip21}, that our main theorem is a finitization in the sense of Tao of the following theorem (which is, again, a somewhat restricted form of the main result in \cite{Rei80}).

\begin{theorem}[\cite{Rei80}]\label{main-thm-sunny}
In the hypotheses of either Theorem~\ref{main-metastab} or Corollary~\ref{cor-metastab}, we have that for all $(\lambda_n) \subseteq (0,\infty)$ such that $\lim\limits_{n \to \infty} \lambda_n = \infty$, the sequence $(J_{\lambda_n A} x)$ converges to a zero of $A$, which we denote by $Qx$. In addition, the map $Q: C \to zer(A)$ thus defined is a sunny nonexpansive retraction (and therefore the unique such one).
\end{theorem}

As in \cite{KohSip21}, the metastability of the sequence $(J_{\lambda_n A} x)$ immediately implies, using almost purely logical principles, its convergence. It is then clear that the limit does not depend on the $(\lambda_n)$, so we can unambiguously dub it $Qx$. For the rest of the proof, we fix a $(\lambda_n)$ and denote, for all $n$, $x_n:=J_{\lambda_n A} x$. That $Qx$ is a fixed point of $J_A$ (and thus a zero of $A$) follows from the continuity of $J_A$ and the fact (proven already in Subsection~\ref{morales}) that
$$\lim_{n \to \infty} \|x_n - J_Ax_n\| = 0.$$

If $x$ is already a zero of $A$, then clearly, for all $n$, $x_n=J_{\lambda_n A} x=x$ and therefore $Qx=x$. We have thus shown that $Q$ is a retraction. To show that $Q$ is sunny and nonexpansive, we seek to apply Proposition~\ref{char-sunny}. Let $p$ be a zero of $A$. We now argue as in Subsection~\ref{morales}. Let $n \in \N$. Since $\left(x_n, \frac{x-x_n}{\lambda_n}\right) \in A$ and $(p,0) \in A$, and $A$ is accretive, we have that
$$\left\langle \frac{x-x_n}{\lambda_n} - 0, j(x_n-p) \right\rangle \geq 0,$$
i.e. that $\langle x_n - x, j(x_n -p)\rangle \leq 0$. Since $j$ is homogeneous, we have that
$$\langle x - x_n, j(p-x_n) \rangle \leq 0.$$
Since the above holds for all $n \in \N$, by passing to the limit, using the continuity of $j$ (together with the fact that $(x_n)$ is bounded), we get that
$$\langle x - Qx, j(p-Qx) \rangle \leq 0,$$
which is what we needed to show.

\mbox{}

Finally, I want to mention some related work in the area. The result in this paper has been applied by Findling and Kohlenbach in \cite{FinKoh} towards obtaining a rate of metastability for an algorithm that finds a zero of an accretive operator in a Banach space which is uniformly convex and uniformly smooth, without relying on its resolvent. In addition, Pinto has recently introduced in \cite{PinXX} a nonlinear generalization of smooth spaces for which he proved the corresponding version of Reich's theorem. A quantitative analysis of that theorem, building on \cite{KohSip21} and the present paper will be presented in the forthcoming paper \cite{PinXY}.

\section{Acknowledgements}

This work originated in and benefited from discussions with Ulrich Kohlenbach, whom I would like to thank.

\end{document}